\documentclass[a4paper,twoside]{amsart}
\usepackage[utf8]{inputenc}
\usepackage[T1]{fontenc}
\usepackage[british]{babel}
\usepackage{amssymb,amsmath,amsthm,amscd}
\usepackage{mathrsfs}
\usepackage{enumerate}
\usepackage{verbatim}
\usepackage{lmodern}
\usepackage{amsmath,amssymb,amsthm,amsbsy}
\usepackage{mathtools} 					
\usepackage{microtype} 					
\usepackage[shortcuts]{extdash} 			
\usepackage{color}
\usepackage{enumerate}
\usepackage{graphicx}    
\usepackage{xspace}					
\usepackage[colorlinks=true,linkcolor=black, citecolor=black]{hyperref}  
\usepackage{mparhack}					
\usepackage{accents}

\usepackage{tikz}
  \usetikzlibrary{arrows,matrix,decorations.markings,patterns,cd}
  \tikzset{
    >=latex,
    bnode/.style={circle,fill=black,draw=black},
    middlearrow/.style={	
	  decoration={markings, mark= at position 0.55 with {\arrow{#1}}},
	  postaction={decorate}
	  },
    double arrow/.style args={#1 colored by #2 and #3}{ 
	  -stealth,line width=#1,#2, 
	  postaction={draw,-stealth,#3,line width=(#1)/3,
                shorten <=(#1)/3,shorten >=2*(#1)/3}, 
	  }
	  
  }
\tikzset{
        hatch distance/.store in=\hatchdistance,
        hatch distance=10pt,
        hatch thickness/.store in=\hatchthickness,
        hatch thickness=2pt
    }
    \makeatletter
    \pgfdeclarepatternformonly[\hatchdistance,\hatchthickness]{flexible hatch}
    {\pgfqpoint{0pt}{0pt}}
    {\pgfqpoint{\hatchdistance}{\hatchdistance}}
    {\pgfpoint{\hatchdistance-1pt}{\hatchdistance-1pt}}%
    {
        \pgfsetcolor{\tikz@pattern@color}
        \pgfsetlinewidth{\hatchthickness}
        \pgfpathmoveto{\pgfqpoint{0pt}{0pt}}
        \pgfpathlineto{\pgfqpoint{\hatchdistance}{\hatchdistance}}
        \pgfusepath{stroke}
    }
    
    \pgfdeclarepatternformonly[\hatchdistance,\hatchthickness]{flexible revhatch}
    {\pgfqpoint{0pt}{0pt}}
    {\pgfqpoint{\hatchdistance}{-\hatchdistance}}
    {\pgfpoint{\hatchdistance-1pt}{\hatchdistance-1pt}}%
    {
        \pgfsetcolor{\tikz@pattern@color}
        \pgfsetlinewidth{\hatchthickness}
        \pgfpathmoveto{\pgfqpoint{0pt}{0pt}}
        \pgfpathlineto{\pgfqpoint{\hatchdistance}{-\hatchdistance}}
        \pgfusepath{stroke}
    }

\newcommand{\R}{\ensuremath{\mathbb{R}}}%
\newcommand{\N}{\ensuremath{\mathbb{N}}}%
\newcommand{\sphere}{\ensuremath{\mathbb{S}}}%

\newcommand{\acts}{\ensuremath{\curvearrowright}}%
\renewcommand\epsilon\varepsilon

\newcommand{\wdist}[1]{{\delta^{#1}_\Gamma}} 
\newcommand{\wLdist}[1]{{\delta^{#1}_\Lambda}} 

\newcommand{\shortacts}{\,\text{\scalebox{0.8}[1]{$\curvearrowright $}}\,}

\newcommand{\CG}{\mathcal{G}}
\newcommand{\CN}{\mathcal{N}}
\newcommand{\CO}{\mathcal{O}}
\newcommand{\CP}{\mathcal{P}}

\newcommand{\NN}{\mathbb{N}}
\newcommand{\RR}{\mathbb{R}}
\newcommand{\ZZ}{\mathbb{Z}}

\newcommand{\Bigmid}{\mathrel{\Big|}}

\DeclarePairedDelimiter\abs{\lvert}{\rvert}	
\DeclarePairedDelimiter\norm{\lVert}{\rVert}	

\theoremstyle{definition}
\newtheorem{thmA}{Theorem}

\newtheorem{thm}{Theorem}[section]
\newtheorem{dfn}[thm]{Definition}

\newtheorem{lem}[thm]{Lemma}
\newtheorem{prop}[thm]{Proposition}
\newtheorem{cor}[thm]{Corollary}

\newtheorem{rem}[thm]{Remark}

\author[de Laat]{Tim de Laat}
\address{Tim de Laat, Westf\"alische Wilhelms-Universit\"at M\"unster, Germany}
\email{tim.delaat@uni-muenster.de}

\author[Vigolo]{Federico Vigolo}
\address{Federico Vigolo, University of Oxford, United Kingdom}
\email{federico.vigolo@maths.ox.ac.uk}

\title{Superexpanders from group actions on compact manifolds}
  
\begin{document}

\begin{abstract}
  It is known that the expanders arising as increasing sequences of level sets of warped cones, as introduced by the second-named author, do not coarsely embed into a Banach space as soon as the corresponding warped cone does not coarsely embed into this Banach space. Combining this with non-embeddability results for warped cones by Nowak and Sawicki, which relate the non-embeddability of a warped cone to a spectral gap property of the underlying action, we provide new examples of expanders that do not coarsely embed into any Banach space with nontrivial type. Moreover, we prove that these expanders are not coarsely equivalent to a Lafforgue expander. In particular, we provide infinitely many coarsely distinct superexpanders that are not Lafforgue expanders. In addition, we prove a quasi-isometric rigidity result for warped cones.
\end{abstract}

\maketitle

\section{Introduction}

Expanders are sequences of finite sparse highly connected graphs with a growing number of vertices (see Section \ref{subsec:expanders} for the definition). It is a well-known result of Gromov \cite{MR1826251} that expanders do not coarsely embed into a Hilbert space or into $\ell^p$, with $1 \leq p < \infty$ (see also \cite{MR1978492,MR1489105}). It is a deep open problem whether there exist expanders that can be embedded coarsely into some superreflexive Banach space (recall that a Banach space is superreflexive if and only if it is isomorphic to a uniformly convex Banach space). In \cite{MR2732331}, Pisier introduced the class of uniformly curved Banach spaces and showed that expanders are not coarsely embeddable into such spaces. This was a great step forward, as there are no known examples of superreflexive Banach spaces that are not uniformly curved. However, the aforementioned problem remains open.

An expander that does not coarsely embed into any uniformly convex space is called a \emph{superexpander} (see also Remark \ref{rem:superexpanders}). An expander is said to be of Margulis type if it arises as a sequence of Cayley graphs of finite quotients of a group $\Gamma$. In relation to his groundbreaking work on the Baum-Connes conjecture, Lafforgue provided the first examples of superexpanders. Indeed, in \cite{MR2423763} (see also \cite{MR2574023}), he introduced a strengthening of property (T), named strong Banach property (T), which implies a very strong obstruction to coarse embeddability of Margulis type expanders coming from groups with this property. Indeed, he proved that $\mathrm{SL}(3,F)$, where $F$ is a non-Archimedean local field, and cocompact lattices in this group have strong Banach property (T), which implies that the Margulis type expanders coming from such lattices do not coarsely embed into any Banach space with nontrivial type (see Section \ref{subsec:geometry.of.Banach.spaces} for the definition of type). In particular, this implies that these expanders are superexpanders, since the class of Banach spaces with nontrivial type is known to be larger than the class of uniformly convex spaces. More examples of groups with strong property (T) or other Banach space versions of property (T) relative to various classes of Banach spaces were provided in \cite{MR3190138,MR3474958,MR3407190,MR3533271,MR3781331,MR3606450}. Other examples of superexpanders, obtained by means of a combinatorial construction, were provided in the groundbreaking work of Mendel and Naor \cite{MR3210176}.

In \cite{MR2115671}, Roe introduced the notion of warped cone. This construction yields a geometric object from an action of a group on a compact metric space. Recently, the coarse geometry of warped cones has attracted an increasing amount of attention (see e.g.~\cite{drutunowak,sawicki1,MR3577880,sawickiwu,wangwang}). In \cite{MR3577880}, Nowak and Sawicki proved that warped cones for which the underlying action has spectral gap are difficult to coarsely embed into Banach spaces. In particular, if the warped cone is constructed from an action of a group with strong Banach property (T), then the warped cone does not coarsely embed into any Banach space with nontrivial type. Nowak and Sawicki also considered actions with spectral gap of groups without property (T).

More recently, the second-named author investigated the approximation of actions on measure spaces by means of finite graphs \cite{vigolo1}. One of his main results is that ``increasing sequences of level sets'' of warped cones constructed from an action that is ``expanding in measure'' are uniformly quasi-isometric to an expander. By combining this result with the aforementioned result of Nowak and Sawicki, one can deduce that expanders arising as the level sets of warped cones constructed from actions of groups with strong Banach property (T) do not coarsely embed into any Banach space with nontrivial type. This result appeared first without proof in a thesis of the second-named author \cite[Remark 7.12]{vigolothesis}, and it was obtained independently by
Sawicki \cite{sawicki2}. For completeness, we will provide a short proof of this result. In fact, we prove a slightly more general non-embeddability result for approximating graphs (see Proposition \ref{prop:uniform.gap.implies.no.embedding}), for which we do not need the warped cone construction.

The main aim of this article is to study a class of examples of expanders that arise as increasing sequences of level sets of warped cones constructed from actions of groups with strong Banach property (T). As an explicit example of such an action, we consider the natural action of $\mathrm{SO}\big(d,\mathbb{Z}[\frac{1}{5}]\big)$ on the compact Riemannian manifold $\mathrm{SO}(d,\mathbb{R})$. It follows from the aforementioned non-embeddability result that these expanders do not coarsely embed into any Banach space with nontrivial type. In particular, they are superexpanders. We also prove that they are not coarsely equivalent to a Lafforgue expander. Our main result can be formulated as follows.
\begin{thmA} \label{thm:infinitelymanysuperexpanders}
  There exist infinitely many coarsely distinct superexpanders that are not coarsely equivalent to a Lafforgue expander.
\end{thmA}
Let us point out that, by cardinality arguments, one can easily construct a continuum of non\=/equivalent superexpanders (Proposition~\ref{prop:cardinality.and.non.coarse.equivalence}). Still, using such ways to artificially create ``inequivalent'' expanders is rather unsatisfactory. More interestingly, we can produce countably many superexpanders that are not even ``locally'' coarsely equivalent (see Remarks \ref{rem:common.subsequences} and \ref{rem:local.qi.rigidity}).

The possibility of constructing new examples of superexpanders from warped cones was already considered by Sawicki \cite{sawicki2}, but he did not provide explicit examples of new superexpanders or prove that the expanders arising in this way are not coarsely equivalent to a Lafforgue expander.

The methods that we use in order to prove that expanders are not coarsely equivalent are inspired by the work of Khukhro and Valette \cite{MR3658731}. On our way towards Theorem \ref{thm:infinitelymanysuperexpanders}, we prove some general facts about coarse equivalences between expanders, and we prove the following quasi\=/isometric rigidity result for warped cones.

\begin{thmA} \label{thm:intro.qi.rigidity}
 Let $\Gamma\curvearrowright M$ and $\Lambda\curvearrowright N$ be essentially free actions by isometries on Riemannian manifolds. If the associated warped cones are quasi\=/isometric then $\Gamma\times\ZZ^{\dim(M)}$ is quasi\-/isometric to $\Lambda\times\ZZ^{\dim(N)}$.
\end{thmA}

We would also like to point out that one can use approximating graphs to obtain expanders whose vertex sets have arbitrary cardinality, as is stated in the following result.

\begin{thmA}\label{thm:intro.arbitrary.cardinality}
 For every increasing sequence $(c_n)$ of natural numbers, there exists a superexpander such that the $n$\=/th graph in the sequence has $c_n$ vertices.
\end{thmA}

After the completion of this article, Fisher, Nguyen and van Limbeek \cite{fishernguyenvanlimbeek} studied the quasi-isometry types of expanders constructed from group actions on homogeneous spaces through the warped cone construction. Their approach gives rise to continua of ``quasi-isometrically disjoint'' superexpanders, meaning that these expanders do not have quasi-isometric subsequences. This broadly generalizes Theorem \ref{thm:infinitelymanysuperexpanders}. The profound rigidity result in \cite{fishernguyenvanlimbeek} implying the difference of the quasi-isometry types of the expanders relies on the computation of their coarse fundamental group. The computation of the coarse fundamental group was independently considered by the second-named author \cite{vigolo2}. In any case, we think that the approach of this article has its own value, mainly because it relies on elementary techniques and because it clarifies some aspects of the local geometry of warped cones.

This article is organized as follows. Section \ref{sec:preliminaries} deals with some background and preliminaries. Our result on the non-embeddability of approximating graphs is proven in Section \ref{sec:nonembeddability}. In Section \ref{sec:actions}, we study examples of actions of groups with strong Banach property (T) on compact Riemannian manifolds, and in Section \ref{sec:coarsegeometry}, we prove that the expanders constructed from these actions are not coarsely equivalent to a Lafforgue expander. Moreover, we prove Theorem \ref{thm:infinitelymanysuperexpanders}, Theorem \ref{thm:intro.qi.rigidity} and Theorem \ref{thm:intro.arbitrary.cardinality} in Section \ref{sec:coarsegeometry} (these are Theorem \ref{thm:warped.cones.not.qi.to.Lafforgue}, Theorem \ref{thm:qi.rigidity.of.warped.cones} and Corollary \ref{cor:expanders.arbitrary.cardinality}, respectively).

\section*{Acknowledgements}
The authors thank Pierre-Emmanuel Caprace, Cornelia Dru\c{t}u, Masato Mimura, Mikael de la Salle and Alain Valette for helpful suggestions and remarks. They also thank Cornelia Dru\c{t}u for comments on a preliminary version and Panagiotis Papazoglou for pointing out an error in a preliminary version.

The authors would like to thank the Isaac Newton Institute for Mathematical Sciences, Cambridge (EPSRC grant no EP/K032208/1), for support and hospitality during the programme ``Non-positive curvature, group actions and cohomology'', where part of the work on this article was undertaken. The first-named author is supported by the Deutsche Forschungsgemeinschaft (SFB 878). The second-named author is supported by the EPSRC Grant 1502483 and the J.T.Hamilton Scholarship.

\section{Preliminaries} \label{sec:preliminaries}
Throughout the paper, $\Gamma$ will always denote a finitely generated group and $S$ a finite symmetric generating set of $\Gamma$. We can assume $S$ to be fixed, as the specific choice of finite generating set will never be relevant for our statements. Actions of countable groups on probability spaces are always assumed to be measure preserving.

The group $\Gamma$ is a (discrete) metric space when equipped with the associated word metric. More generally, we always think of (the set of vertices of) a connected graph as a discrete metric space using the shortest path distance. The word metric on $\Gamma$ coincides with the distance coming from the identification of $\Gamma$ with its Cayley graph.

\subsection{Expanders and superexpanders} \label{subsec:expanders}
Let $\mathcal{G}=(V,E)$ be a finite connected graph. For $A \subset V$, let $\partial A$ be the set of edges in $E$ joining an element of $A$ with an element of $V \setminus A$. The \emph{Cheeger constant} $h(\mathcal{G})$ of $\mathcal{G}$ is defined as
\[
  h(\mathcal{G})=\min\left\{\frac{|\partial A|}{|A|} \ \middle|\ A \subset V,\;|A| \leq \frac{1}{2}|V|\right\}.
\]
The Cheeger constant is a measure of the connectivity of a graph.

Let $(\mathcal{G}_n)$ be a sequence of finite connected graphs with degrees bounded by a constant $D$ independent of $n$ (the \emph{degree} of a graph is the maximum of the degrees of its vertices). Then the sequence $(\mathcal{G}_n)$ is said to be an \emph{expander} if $\lim_{n \to \infty} |V_n| = \infty$ (where $V_n$ is the vertex set of $\mathcal{G}_n$) and $\exists C > 0$ such that $h(\mathcal{G}_n) \geq C$ for all $n$. When this is the case we also say that the graphs $\CG_n$ form a family of expanders.\\

If $\Lambda$ is a cocompact lattice in an almost simple higher rank algebraic group over a non-Archimedean local field (note that in this case $\Gamma$ has strong Banach property (T) by \cite{MR2423763,MR2574023,MR3190138}), $(\Lambda_n)$ is a nested sequence of finite index normal subgroups of $\Lambda$ with trivial intersection and $\vert \Lambda / \Lambda_n \vert \to \infty$, and $S'$ is a generating set of $\Lambda$, then the sequence of Cayley graphs ${\rm Cay}(\Lambda/\Lambda_n,\pi_n(S'))$ is a superexpander \cite{MR2423763}. We call such expanders \emph{Lafforgue expanders}.

\begin{rem} \label{rem:superexpanders}
In some texts, a superexpander is defined as a sequence of finite connected $D$-regular graphs $\mathcal{G}_n=(V_n,E_n)$, $n \in \mathbb{N}$, such that $\lim_{n \to \infty} |V_n| = \infty$ and such that for every superreflexive Banach space $X$ there exists a constant $\gamma > 0$ such that for all $n \in \mathbb{N}$ and $f\colon V_n \to X$, the inequality
\[
	\frac{1}{|V_n|^2}\sum_{v,w \in V_n}\|f(v)-f(w)\|^2 \leq \frac{\gamma}{D|V_n|} \sum_{(v,w) \in E_n} \|f(v)-f(w)\|^2
\]
holds. This definition is stronger than the definition that we use in this article. However, from \cite{sawicki2}, it follows that the superexpanders that we consider are also superexpanders according to the definition in this remark (with the only minor caveat that these graphs need not be $D$\=/regular).
\end{rem}

\subsection{Coarse embeddings and quasi-isometries}\label{subsec:coarse.embeddings}
A map $f\colon (X,d_X)\to (Y,d_Y)$ between metric spaces is called a \emph{coarse embedding} if there are two increasing and unbounded control functions $\rho_-,\rho_+\colon [0,\infty)\to[0,\infty)$, such that
\[
 \rho_-\left(d_X(x_1,x_2)\right)\leq d_Y\left(f(x_1),f(x_2)\right)\leq \rho_+\left(d_X(x_1,x_2)\right)
\]
for every $x_1,x_2 \in X$.\\

A coarse embedding $f\colon (X,d_X)\to (Y,d_Y)$ is called a \emph{coarse equivalence} if there exists a constant $C>0$ such that for every $y \in Y$, there exists an $x \in X$ with $d_Y(y,f(x)) \leq C$.  Equivalently, a coarse embedding $f$ is a coarse equivalence if and only if there exists a coarse embedding $g\colon(Y,d_Y)\to(X,d_X)$ so that both $f\circ g$ and $g\circ f$ are at bounded distance from the identity function. \\

A family of metric spaces $(X_n,d_{X_n})$ coarsely embeds into $(Y,d_Y)$ if there are coarse embeddings $f_n \colon (X_n,d_{X_n}) \to (Y,d_Y)$ with the same control functions $\rho_-$ and $\rho_+$.\\

A map $f\colon (X,d_X)\to (Y,d_Y)$ between metric spaces is called a \emph{quasi-isometry} if there exist constants $L \geq 0$ and $A \geq 0$, such that
\[
  \frac{1}{L}d_X(x_1,x_2)-A \leq d_Y(f(x_1),f(x_2)) \leq Ld_X(x_1,x_2)+A
\]
for every $x_1,x_2 \in X$, and if for every $y \in Y$, there exists an $x \in X$ with $d_Y(y,f(x)) \leq A$. In this case we say that $f$ is an $(L,A)$\=/quasi\=/isometry.\\

Two spaces are \emph{quasi-isometric} if there exists a quasi-isometry between them. Two families $(X_n,d_{X_n})$, $(Y_n,d_{Y_n})$ of metric spaces are \emph{uniformly quasi-isometric} if there exist quasi-isometries $f_n\colon(X_n,d_{X_n}) \to (Y_n,d_{Y_n})$ with the same constants $L,A$.

Note that two connected graphs $\CG$ and $\CG'$ (and more generally, two geodesic spaces) are quasi\=/isometric if and only if they are coarsely equivalent. We will hence use the terms quasi\=/isometric and coarsely equivalent as synonyms.

\subsection{Separated subsets of metric spaces}\label{ssec:separated.subset}
Given $r>0$, we say that a subset $Y$ of a metric space $(X,d)$ is \emph{$r$\=/separated} if $d(y,y')\geq r$ for every two $y,y'\in Y$ satisfying $y\neq y'$. Note that if $Y$ is $r$\=/separated, then the open balls $B(y,\frac{r}{2})$ with $y\in Y$ are all disjoint in $X$. 

We say that $Y$ is \emph{$r$\=/dense} if the union of all the balls $B(y,r)$ with $y\in Y$ covers the whole of $X$. Note that an $r$\=/separated set is maximal (with respect to the ordering given by the inclusion) if and only if it is also $r$\=/dense. In particular, in every metric space there exist subsets that are $r$\=/separated and $r$\=/dense. 

It is easy to verify that if $Y$ is an $r$\=/separated and $R$\=/dense subset of $(X,d_X)$ and $f\colon (X,d_X)\to (X',d_{X'})$ is an $(L,A)$\=/quasi\=/isometry, then $f(Y)$ is $(\frac{r}{L}-A)$\=/separated and $(RL+2A)$\=/dense in $X'$.

\subsection{Small balls in Riemannian manifolds}\label{ssec:small.balls.in.mflds}
Let $(M,d)$ be a compact Riemannian manifold. Then for every $r>0$ we let $v_M(r)$ and $V_M(r)$ denote the minimum and maximum Riemannian volume of a ball of radius $r$ in $M$: 
\[
v_M(r)\coloneqq \min\{{\rm Vol}\big(B(x,r)\big)\mid x\in M\} \qquad V_M(r)\coloneqq \max\{{\rm Vol}\big(B(x,r)\big)\mid x\in M\}. 
\]

Note that if $Y\subset M$ is $r$\=/separated, then ${\rm Vol}(M)\geq v_M(\frac{r}{2})\abs{Y}$. If $Y$ is $R$\=/dense, then ${\rm Vol}(M)\leq V_M(R)\abs{Y}$. \\

Using the Bishop-Gunther-Gromov Volume Comparison Theorem (see e.g.~\cite{MR2088027}), we can deduce the following lemma.
\begin{lem}\label{lem:volume.comparison.balls}
Let $(M,d)$ be a compact Riemannian manifold. For every $\kappa\geq 1$, there exists a constant $C(\kappa)$ such that
\[
V_M(\kappa r)\leq C(\kappa)v_M(r)
\]
for every $r>0$.
\end{lem}

It follows that if $Y$ is $r$\=/separated and $R$\=/dense, then 
\[
\abs{Y}\leq \frac{{\rm Vol}(M)}{v_M(\frac{r}{2})}\leq C\Big(\frac{2R}{r}\Big)\abs{Y}.
\]

\subsection{Banach spectral gap}\label{subsec:Banach.gap}
Let $(X,\nu)$ be a probability space and $E$ a Banach space. The \emph{Bochner space} $B=L^2(X,\nu;E)$ is the Banach space of square-integrable functions $f\colon X \to E$ equipped with the norm
\[
 \|f\|_B=\left(\int_X \|f(x)\|_E^2\ d\nu(x)\right)^{\frac{1}{2}}.
\]
We let $L^2_0(X,\nu;E)$ be the subspace of $B$ of functions with zero mean. Any measure preserving action $\rho\colon\Gamma\curvearrowright (X,\nu)$ induces a unitary action on $B$ and on $L^2_0(X,\nu;E)$ by $\gamma\cdot f(x) = f(\gamma^{-1}\cdot x)$. We say that $\rho$ has \emph{$E$-spectral gap} if there exists a constant $\varepsilon>0$ so that for every $f\in L^2_0(X,\nu;E)$, we have $\max_{s \in S}\norm{f-s \cdot f}_B\geq\varepsilon\norm{f}_B$.

\subsection{Approximating graphs}\label{subsec:approx.graphs}
Let $(X,\nu)$ be as above. Denote by $\CP$ a partition of $X$ into finitely many disjoint measurable sets (which we call \emph{regions}). Given a constant $Q\geq 1$, we say that $\CP$ has \emph{$Q$\=/bounded measure ratios} if 
\[
 \frac{\nu(R)}{\nu(R')}\leq Q
\]
for every $R,R'\in\CP$. 

Given an action $\Gamma\curvearrowright (X,\nu)$ and a partition $\CP$, the associated \emph{approximating graph} is the graph $\CG_\CP(\Gamma\curvearrowright X)$ whose vertices are the regions in $\CP$ and two regions $R,R'\in\CP$ are connected by an edge if there exists an $s\in S$ so that $\nu\big((s\cdot R)\cap R'\big)>0$. 

If $X$ is also a metric space, then the \emph{mesh} of $\CP$ is defined as the maximum of the diameters of the regions in $\CP$. For more details on approximating graphs, we refer to \cite{vigolo1}.

\subsection{Warped cones}\label{subsec:warped.cones}
Warped cones were introduced by Roe \cite{MR2115671}. We will work with a slightly different definition, which produces families of spaces that are quasi-isometric to the level sets of the original warped cones.

Let $M$ be a compact Riemannian manifold with Riemannian metric $d$, and for every $t \geq 1$, let $d^t=td$ denote the rescaling of the Riemannian metric. As in \cite{vigolo1}, given an action $\rho\colon\Gamma\curvearrowright(M,d)$ by Lipschitz homeomorphisms, we define the \emph{$t$-level} of the \emph{warped cone} as the metric space $(M,\wdist{t})$ where $\wdist{t}$ is the \emph{warped distance}, i.e.~the largest metric satisfying
\begin{itemize}
 \item $\wdist{t}(x,y)\leq d^t(x,y)$ for every $x,y\in M$;
 \item $\wdist{t}(x,s\cdot x)\leq 1$ for every $x\in M$ and $s\in S$.
\end{itemize}
Note that the definition depends on the choice of the generating set $S$. Still, different generating sets produce quasi-isometric warped cones.

In \cite{MR2115671}, the warped cone was defined as the metric space $\CO_\Gamma(M)=(M\times[1,\infty),\wdist{})$ for some specific metric $\wdist{}$. According to our definition, the identifications $(M,\wdist{t})\to M\times \{t\}\subset \CO_\Gamma(M)$ are not isometries, but they are uniform quasi-isometries (see \cite[Lemma 6.5]{vigolo1}). 

The next result follows from combining \cite[Proposition 1.10]{MR2115671} and \cite[Section 6]{vigolo1}.
\begin{prop}
 The $t$-level sets $(M,\wdist{t})$ are uniformly quasi\=/isometric to graphs $\CG_t(\Gamma\shortacts M)$ with uniformly bounded degrees. The graphs $\CG_t(\Gamma\shortacts M)$ are graphs approximating the action $\rho$ with respect to some measurable partitions of mesh $\approx \frac{1}{t}$ and uniformly $Q$\=/bounded measure ratios.
\end{prop}

\subsection{On the geometry of Banach spaces}\label{subsec:geometry.of.Banach.spaces}
A Banach space $X$ is called \emph{uniformly convex} if $\delta_X(\varepsilon)>0$ for all $\varepsilon \in (0,2]$, where
\[
  \delta_X(\varepsilon) = \inf\left\{ 1 - \frac{\|x+y\|}{2} \;\Bigg\vert\; \|x\| \leq 1,\;\|y\| \leq 1,\;\|x-y\| \geq \varepsilon \right\}.
\]
Let $(g_i)_{i \in \N}$ be a countable family of independent complex Gaussian $\mathcal N(0,1)$ random variables on a probability space $(\Omega,\mathbb P)$. A Banach space $X$ is said to have \emph{type $p \geq 1$} if there exists a constant $T>0$ such that for every $n \in \mathbb{N}$ and $x_1,\dots,x_n \in X$, we have $\| \sum_i g_i x_i \|_{L^2(\Omega;X)} \leq T \left(\sum_i \|x_i\|^p\right)^{1/p}$.\\

Uniformly convex Banach spaces have nontrivial type, i.e.~type $>1$. The converse is known to be false. For more details on type, we refer to \cite{MR1999197}.

\subsection{Strong property (T)}\label{subsec:strong.T}
Strong Banach property (T) was introduced by Lafforgue \cite{MR2423763,MR2574023}. In this article, we use the version of strong property (T) relative to classes of Banach spaces, which is implicit in \cite{MR2574023}, and which appeared explicitly in \cite{MR3474958}. The latter article also gives a characterization of strong property (T), which is what we take as its definition in this article.

\begin{dfn}\label{def=strongT}
Let $\mathcal{E}$ be a class of Banach spaces. A locally compact group $G$ has \emph{strong property} (T) with respect to $\mathcal{E}$, denoted by (T$^{\mathrm{strong}}_{\mathcal{E}}$), if for every length function $\ell$ on $G$ there exists a sequence of compactly supported symmetric Borel measures $m_n$ on $G$ such that for every Banach space $E$ in $\mathcal E$ there exists a constant $t>0$ such that the following holds: for every strongly continuous representation $\pi\colon G \rightarrow B(E)$ satisfying $\|\pi(g)\|_{B(E)} \leq L e^{t \ell(g)}$ with $L\in\R_+$, the sequence $\pi(m_n)$ converges in the norm topology on $B(E)$ to a projection onto the $\pi(G)$-invariant vectors in $E$.
\end{dfn}
Strong Banach property (T) of Lafforgue corresponds to taking $\mathcal{E}$ to be the class of Banach spaces with nontrivial type.\\

Recall that a measure preserving action of a countable group on a probability space is \emph{ergodic} if and only if every invariant subset of the space on which the group acts has measure $0$ or $1$. We will use the following result on ergodic actions of groups with strong Banach property (T).
\begin{prop} \label{prp:ergodicactions}
 If $\Gamma$ is a finitely generated group with strong Banach property (T) of Lafforgue, then every ergodic measure preserving action $\rho\colon\Gamma\curvearrowright (X,\nu)$ has $E$\=/spectral gap for every Banach space $E$ with non\=/trivial type. Moreover, this spectral gap is uniform in the class of Banach spaces with nontrivial type, i.e.~the constant $\varepsilon$ in the definition of spectral gap in Section \ref{subsec:Banach.gap} does not depend on $E$ nor on the action.
\end{prop}
This result is well known to experts. It follows from the fact that strong property (T) relative to a class $\mathcal{E}$ implies property (T$_E$) uniformly for $E \in \mathcal{E}$. By definition, the latter property implies uniform $E$-spectral gap. For the specific case of classes of uniformly convex Banach spaces and isometric representations on these spaces, the result also follows from \cite[Theorem 4.6]{drutunowak}.

\section{The non-embeddability result} \label{sec:nonembeddability}
Let $\Gamma$ be a finitely generated group with finite symmetric generating set $S$, let $\Gamma\curvearrowright(X,\nu)$ be a measure preserving action on a probability space, and let $E$ be a Banach space. A function $\hat f\colon \CG_\CP(\Gamma\curvearrowright X)\to E$ (here $\hat f$ is defined on the vertices of the graph) naturally induces a function $f\colon X \to E$ assigning to a point $x\in X$ the value of the region $R_x\in\CP$ containing it.
\begin{lem}\label{lem:upper.bound}
 If $\hat f$ is a coarse embedding with control functions $\rho_-$ and $\rho_+$, then for every $s\in S$ we have $\norm{f-s\cdot f}_B\leq \rho_+(1)$.
\end{lem}
\begin{proof}
 Note that for $\nu$\=/almost every $x\in X$ if we let $R_x,R_{s\cdot x}\in\CP$ be the regions containing $x$ and $s\cdot x$, respectively, then $d(R_x,R_{s\cdot x})\leq 1$. Therefore, we have
 \begin{align*}
  \norm{f-s^{-1}\cdot f}_B^2 &=\int_X \norm{f(x)-f(s\cdot x)}_E^2\ d\nu(x)	\\
    &=\int_X \norm{\hat f\left(R_x\right)-\hat f\left(R_{s\cdot x})\right)}_E^2\ d\nu(x) \\
    &\leq \int_X \left(\rho_+(1)\right)^2\ d\nu=\left(\rho_+(1)\right)^2 .
 \end{align*}
\end{proof}
\begin{lem}\label{lem:lower.bound}
 If $\CP$ has $Q$\=/bounded measure ratios, and if the degree of $\CG_\CP(\Gamma\curvearrowright X)$ is $D$, and if $\hat f$ is a coarse embedding with control functions $\rho_-$ and $\rho_+$, then (when the right-hand side is defined) we have
 \[
  \norm{f}_B\geq \frac{1}{4}\,\rho_-\!\!\left(\frac{1}{\log(D)}\log\left(\frac{\abs{\CP}}{2Q}\right)-1\right).
 \]
\end{lem}
\begin{proof}
 Let $C=\norm{f}_B$. Note that the set $X_{2C}=\{x\in X\mid \norm{f(x)}_E\leq 2C\}$ has measure $\nu(X_{2C})> \frac{1}{2}$.
 
 For any $r\geq 0$ and $R\in \CP$, let $\CN_r(R)\subseteq X$ denote the union of all the regions $R'\in\CP$ with $d(R,R')\leq r$. Then it follows from our hypotheses that 
 \[
  \nu\left(\CN_r(R)\right)=\frac{\nu\left(\CN_r(R)\right)}{\nu(X)}\leq Q\frac{D^{r+1}}{\abs{\CP}}.
 \]
 In particular, if we let 
 \[
  r = \frac{1}{\log(D)}\log\left(\frac{\abs{\CP}}{2Q}\right)-1,
 \]
 then $\nu(\CN_r(R))\leq \frac{1}{2}$.
 
 Choose any region $R\subseteq X_{2C}$. By construction, there must exist another region $R'\subseteq X_{2C}$ with $d(R,R')>r$. Therefore, we have
 \[
  \rho_-(r)\leq\rho_-\left(d(R,R')\right)\leq \norm{\hat f(R)-\hat f(R')}_E
  \leq\norm{\hat f(R)}_E+\norm{\hat f(R')}_E \leq 4C, 
 \]
 whence the required inequality. 
\end{proof}
Now, let $\rho_n\colon\Gamma_n\curvearrowright (X_n,\nu_n)$ be a sequence of probability measure preserving actions, let $\CP_n$ be measurable partitions of $X_n$, all with $Q$\=/bounded measure ratios and such that all the approximating graphs $\CG_{\CP_n}(\Gamma_n \curvearrowright X_n)$ have degree at most $D$. Combining Lemma \ref{lem:upper.bound} and Lemma \ref{lem:lower.bound} immediately yields the following result.
\begin{prop}\label{prop:uniform.gap.implies.no.embedding}
 If the actions $\rho_n\colon\Gamma_n\curvearrowright (X_n,\nu_n)$ all have $E$-spectral gap with a uniform constant $\varepsilon>0$, then the approximating graphs $\CG_{\CP_n}(\Gamma_n \curvearrowright X_n)$ do not coarsely embed into $E$ uniformly.
\end{prop}
\begin{rem}
 Since Cayley graphs can be realized as approximating graphs, we can recover as a corollary the well\=/known fact that Lafforgue expanders do not coarsely embed into any Banach space with non\=/trivial type.
\end{rem}

\section{Actions with spectral gap} \label{sec:actions}
Let $M$ be a compact Riemannian manifold, $\rho\colon\Gamma\acts M$ an action by Lipschitz homeomorphisms that preserve the Riemannian metric, and $(t_n)$ an increasing diverging sequence in $[1,\infty)$. It was proved in \cite{vigolo1} that the graphs $\CG_{t_n}(\Gamma\shortacts M)$ quasi\=/isometric to the $t_n$\=/level sets $(M,\wdist{t_n})$ form a family of expanders if and only if $\rho$ has $\RR$-spectral gap. Proposition \ref{prop:uniform.gap.implies.no.embedding} directly implies the following strengthening.
\begin{thm}
 If the action $\rho\colon\Gamma\curvearrowright M$ has $E$-spectral gap, then for every increasing diverging sequence $(t_n)$ in $[1,\infty)$, the graphs $\CG_{t_n}(\Gamma\shortacts M)$ form an expander that does not coarsely embed into $E$.
\end{thm}
We now give examples. For this, we use the approach that Margulis used to solve the Banach-Ruziewicz problem \cite{MR596890}.

For $d \geq 5$, let $\Gamma_d$ consist of matrices in $\mathrm{SO}(d,\mathbb{R})$ whose entries are elements of $\mathbb{Z}[\frac{1}{5}]$ (the subring of $\mathbb{Q}$ generated by the element $\frac{1}{5}$). Consider the diagonal embedding of $\Gamma_d$ into $G_d=\mathrm{SO}(d,\mathbb{Q}_5) \times \mathrm{SO}(d,\mathbb{R})$. Then $\Gamma_d$ is a cocompact lattice in $G_d$ \cite{MR0202718}.

Since $\mathrm{SO}(d,\mathbb{Q}_5)$ is an almost simple algebraic group of higher rank (whenever $d \geq 5$) (see \cite{MR596890}), it has Lafforgue's strong Banach property (T) (see \cite{MR3190138}). This implies that also $G_d$ has strong Banach property (T), since $\mathrm{SO}(d,\mathbb{R})$ is a compact group. Lafforgue proved that strong Banach property (T) passes to cocompact lattices \cite{MR2423763}, which implies that $\Gamma_d$ has strong Banach property (T).

The compact Lie group $\mathrm{SO}(d,\R)$ can be made into a compact Riemannian manifold of dimension $\frac{d(d-1)}{2}$ by choosing any left\=/invariant Riemannian metric on it. The group $\Gamma_d$ is in fact a dense subgroup of $\mathrm{SO}(d,\mathbb{R})$ (recall that we consider the diagonal embedding of $\Gamma_d$ into $G_d$), so $\Gamma_d$ acts ergodically and by isometries on $\mathrm{SO}(d,\R)$ by left multiplication. This observation directly implies the following result.
\begin{thm}\label{thm:expanders.from.SO}
 Let $d\geq 5$ be fixed, and let $(t_n)$ be an increasing diverging sequence. Then the graphs $\CG_{t_n}(\Gamma_d\shortacts\mathrm{SO}(d,\RR))$ approximating the $t_n$\=/levels of the warped cone of $\Gamma_d\acts\mathrm{SO}(d,\RR)$ form an expander that does not coarsely embed into any Banach space with nontrivial type. In particular, they form a superexpander.
\end{thm}
\begin{rem}
Since the natural action $\mathrm{SO}(d,\RR)\acts\sphere^{d-1}\subset \R^d$ by isometries on the unit sphere restricts to an ergodic action of $\Gamma_d$, we can also deduce that for $d\geq 5$ fixed and $(t_n)$ an increasing diverging sequence, the graphs $\CG_{t_n}(\Gamma_d\shortacts\sphere^{d-1})$ approximating the $t_n$\=/levels of the warped cone of $\Gamma_d\acts\sphere^{d-1}$ form an expander that does not coarsely embed into any Banach space with nontrivial type. In particular, they form a superexpander.
\end{rem}

\section{On the coarse geometry of expanders} \label{sec:coarsegeometry}
In this section, we study expanders up to coarse equivalence. As a first remark, we wish to point out that once an expander $(\CG_n)$ is given, one can construct a continuum of non-coarsely equivalent expanders by carefully choosing subsequences of it. Similar arguments are used in \cite[Theorem 2.8]{MR3640615} and \cite[Proposition 2]{MR3658731}.

\begin{prop}\label{prop:cardinality.and.non.coarse.equivalence}
 Let $(\CG_n)$ be a family of finite graphs with uniformly bounded degree and $\abs{\CG_n}\to\infty$. Then there exists a continuum $\mathcal{I}$ of subsets $I_a\subset \NN$ such that for every pair $I_a\neq I_b$ in $\mathcal{I}$, the subsequences $(\CG_n)_{n\in I_a}$ and $(\CG_n)_{n\in I_b}$ are not uniformly coarsely equivalent. 
\end{prop}
\begin{proof}
 Choosing a subsequence if necessary, we can assume that $\abs{\CG_{n+1}}> n\abs{\CG_n}$ for every $n\in\NN$. We claim that for every choice of control functions $\rho_-$ and $\rho_+$ there is an $n_0$ large enough so that for all $n>m>n_0$, the graphs $\CG_n$ and $\CG_m$ cannot be coarsely equivalent with control functions $\rho_-$ and $\rho_+$. Indeed, suppose that there exists such a coarse equivalence $f \colon \CG_n\to\CG_m$, and let $r>0$ be large enough, so that $\rho_-(r)\geq 1$. Then the pre\=/image $f^{-1}(v)$ of any vertex $v\in\CG_m$ must have diameter at most $r$, and it follows that $f^{-1}(v)$ has cardinality at most $D^{r+1}$, where $D$ is the uniform bound on the degree. In particular, we have $m\abs{\CG_m}<\abs{\CG_n}\leq D^{r+1}\abs{\CG_m}$. Hence, to prove the claim, it is sufficient to let $n_0=D^{r+1}$.
 
 It follows from the above discussion that if $I$ and $J$ are two subsets of $\NN$ so that $I \setminus J$ is infinite, then the sequences $(\CG_n)_{n\in I}$ and $(\CG_n)_{n\in J}$ are not uniformly coarsely equivalent. To conclude the proof, it is enough to observe that there exists an uncountable family of sets $I_a\subset \NN$ so that $I_a \setminus I_b$ is infinite for every $a\neq b$. 
\end{proof}

\begin{rem} \label{rem:common.subsequences}
 One can effectively use Proposition \ref{prop:cardinality.and.non.coarse.equivalence} to produce a continuum of non-coarsely equivalent expanders. Moreover, with extra care it is also possible to find a continuum of infinite subsets $I_a\subset \NN$ such that $I_a\cap I_b$ is finite for every $a\neq b$. In turn, this produces a continuum of expanders that do not even admit coarsely equivalent subsequences. 
 
Still, this technique is rather unsatisfactory, as it is quite meaningless to compare families of graphs whose sets have very different cardinalities. Moreover, one cannot trivially use Proposition \ref{prop:cardinality.and.non.coarse.equivalence} to produce examples of non-Lafforgue superexpanders, as a subsequence of a Lafforgue expander is still a Lafforgue expander.
\end{rem}

In view of Proposition \ref{prop:cardinality.and.non.coarse.equivalence}, the following fact may be interesting on its own.

\begin{prop}
 Let $\rho\colon\Gamma\curvearrowright M$ be an action by Lipschitz homeomorphisms on a compact Riemannian manifold. Then there exists a constant $C$ depending only on $M$ such that for every $N\in\NN$ there exists a $t\in\RR$ for which
 \[
  \frac{N}{C}\leq \abs{\CG_t(\Gamma\shortacts M)}\leq CN.
 \]
\end{prop}
\begin{proof}
 The graphs $\CG_t(\Gamma\shortacts M)$ are constructed in \cite{vigolo1} using a one-to-one correspondence between the vertices of $\CG_t(\Gamma\shortacts M)$ and a $\frac{1}{t}$\=/dense, $\frac{1}{t}$\=/separated subset $Y\subset M$. Since $M$ is a compact Riemannian manifold, it follows from the discussion in Section \ref{ssec:small.balls.in.mflds} that there exists a constant $C$ independent of $t$ so that 
 \[
  \frac{1}{C}t^d\leq\abs{Y}\leq Ct^d,
 \]
 where $d$ is the dimension of $M$. Letting $t\coloneqq \sqrt[d]{N}$ yields the result.
\end{proof}

\begin{rem}
 By a more careful argument, one can show that for every $N\in\NN$ it is possible to construct an approximating graph $\CG_t(\Gamma\shortacts M)$ so that $\abs{\CG_t(\Gamma\shortacts M)}=N$. A straightforward way to prove this is to note that one can construct expanders where the vertex set of $\CG_t(\Gamma\shortacts M)$ is in correspondence with an arbitrary  $\frac{1}{t}$\=/dense, $\frac{1}{2t}$\=/separated subset $Y\subset M$ (see the proof of \cite[Theorem 5.5]{vigolo1}). Moreover, if $Y_1$ is $\frac{1}{t}$\=/dense and $\frac{1}{t}$\=/separated and $Y_2\supset Y_1$ is $\frac{1}{2t}$\=/dense and $\frac{1}{2t}$\=/separated then every intermediate subset $Y_1\subseteq Y\subseteq Y_2$ is $\frac{1}{t}$\=/dense and $\frac{1}{2t}$\=/separated. Letting $Y_1\subset Y_2\subset Y_3\subset \cdots$ be a nested sequence of subsets where the $n$\=/th set is $\frac{1}{2^nt}$\=/dense and $\frac{1}{2^nt}$\=/separated (such a sequence exists by Zorn's Lemma), one can choose intermediate subsets of arbitrary (finite) cardinalities.
\end{rem}

Combining the discussion above with Theorem~\ref{thm:expanders.from.SO} we obtain Theorem~\ref{thm:intro.arbitrary.cardinality} as a corollary.

\begin{cor}\label{cor:expanders.arbitrary.cardinality}
	It is possible to construct superexpanders whose graphs have arbitrary cardinality.
\end{cor}

We now wish to show that the examples of superexpanders that we gave in the previous section are never coarsely equivalent to a Lafforgue expander. To do so, we first need to introduce some notation. 

In what follows, we will denote the word length of an element of $\gamma\in\Gamma$ by $\abs\gamma$ and we will equip $\Gamma$ with the \emph{right} word metric, i.e.~we define the distance between two elements $\gamma,\gamma'$ to be $\abs{\gamma'\gamma^{-1}}$.
To avoid confusion, we will generally use the notation $B_{(X,d)}(x,r)$ to denote the ball of radius $r$ centered at the point $x$ in the metric space $(X,d)$. An exception will be the ball $B_{(\Gamma,\abs{\cdot})}(e,r)$, which we will simply denote by $B_\Gamma(r)$. If $Y$ is a subset of $X$, we denote by $N_{(X,d)}(Y,r)$ its neighbourhood of radius $r$:
\[
 N_{(X,d)}(Y,r)\coloneqq \bigcup_{y\in Y}B_{(X,d)}(y,r).
\] 

Given an action $\Gamma\curvearrowright X$, we denote by $B_\Gamma(r)\cdot Y$ the union of the images of $Y$ under the elements in $B_\Gamma(r)$:
\[
B_\Gamma(r)\cdot Y\coloneqq \bigcup\{\gamma\cdot Y\mid \gamma\in\Gamma,\ \abs{\gamma} < r\}.
\]
We define the set $\chi_\Gamma^t(r)\subset X$ as follows:
\[
\chi_\Gamma^t(r)\coloneqq \Big\{x\in X\Bigmid \exists \gamma\in\Gamma,\ \abs{\gamma}\leq 6r\text{ such that }d(x,\gamma \cdot x)\leq\frac{6r}{t}\Big\}.
\]
If we consider a warped cone defined through an action $\Gamma \curvearrowright X$ by isometries, then it is easy to verify (see also \cite{sawicki1}) that the warped distance between two points $x,y\in (X,\wdist{t})$ can be expressed as
 \begin{equation}\label{eq:warped.distance}
  \wdist{t}(x,y)=\inf_{\gamma \in \Gamma}\big[d^t(x,\gamma\cdot y)+\abs{\gamma}\big]
    =\inf_{\gamma \in \Gamma}\big[t d(x,\gamma\cdot y)+\abs{\gamma}\big].
 \end{equation}
In this case, we have the following inclusion of neighbourhoods: 
\[
N_{(M,\wdist{t})}(Y,r)\subseteq B_\Gamma(r)\cdot N_{(M,d)}\Big(Y,\frac{r}{t}\Big).
\]

Now, let $(M,d)$ be a compact Riemannian manifold and $\rho\colon\Gamma\curvearrowright M$ an action by isometries. Equip the direct product $\Gamma\times\RR^k$ with the $\ell^1$\=/distance $\abs{(\gamma,v)}\coloneqq \abs{\gamma}+\norm{v}_2$. The following lemma characterizes balls in warped cones up to bi-Lipschitz equivalence. Similar observations were made in \cite[Lemma 3.8]{sawickiwu}.
\begin{lem}\label{lem:balls.lipschitz.to.Euclidean}
Let $L>1$ and $r>0$ be fixed. Then there exists a $t_0$ large enough so that for every $t\geq t_0$ and for every $x_0\notin \chi_\Gamma^t(r)$, the ball $B_{(M,\wdist{t})}(x_0,r)$ is $L$-bi-Lipschitz equivalent to the ball of radius $r$ in the product $\Gamma\times \RR^{\dim(M)}$.
\end{lem}
\begin{proof}
Fix $t>1$ and $x_0\notin \chi_\Gamma^t(r)$. By definition of $\chi_\Gamma^t(r)$, it follows that the balls $B_{(M,d^t)}(\gamma \cdot x_0,3r)$ with $\gamma\in B_\Gamma(3r)$ are disjoint. Note that the image $\gamma\cdot B_{(M,d^t)}(x_0,3r)$ coincides with the ball $B_{(M,d^t)}(\gamma\cdot x_0,3r)$.

Since $M$ is compact, the infimum in the equality \eqref{eq:warped.distance} is actually a minimum. Therefore, for every two points $x,y\in B_{(M,\wdist{t})}(x_0,r)$, there exist $\gamma_x,\gamma_y\in B_\Gamma(r)$ so that $\wdist{t}(x,x_0)=\abs{\gamma_x}+d^t(\gamma_x \cdot x_0,x)$ and $\wdist{t}(y,x_0)=\abs{\gamma_y}+d^t(\gamma_y \cdot x_0,y)$. Again by \eqref{eq:warped.distance}, there exists a $\gamma\in\Gamma$ with $\abs{\gamma}\leq 2r$ so that 
\[
 2r\geq\wdist{t}(x,y)=\abs{\gamma}+d^t(\gamma \cdot x,y).
\]
It follows that the point $y$ belongs to both $\gamma\gamma_x \cdot B_{(M,d^t)}(x_0,3r)$ and $\gamma_y\cdot B_{(M,d^t)}(x_0,r)$ and therefore, by construction, we must have $\gamma=\gamma_y\gamma_x^{-1}$. As a consequence, it is easy to deduce that the ball of radius $r$ centred at $x_0$ in $(M,\wdist{t})$ is actually isometric to the ball of radius $r$ centred at $(x_0,e)$ in the direct product $(M,d^t)\times\Gamma$ equipped with the $\ell^1$\=/distance.

Recall that the differential $\mathrm{d}_0\exp\colon T_{x}M\to T_{x}M$ of the Riemannian exponential map $\exp\colon T_{x}M\to M$ at the point $0\in T_{x}M$ is the identity. For every $\varepsilon>0$, there exists a $\eta(\varepsilon)$ small enough so that the restriction of $\exp$ to $B_{T_{x} M}(0,\eta)$ is a $(1+\varepsilon)$\=/bi-Lipschitz map onto $B_{(M,d)}(x,\eta)$ when $0 < \eta < \eta(\varepsilon)$ and, by compactness, we can assume that the same constant $\eta(\varepsilon)$ does the job for every point $x\in M$. 

Now, since the ball $B_{(M, d^t)}(x_0,r)$ is isometric to the ball $B_{(M,d)}(x_0,\frac{r}{t})$ with the metric rescaled by $t$, we deduce that $B_{(M, d^t)}(x_0,r)$ is $(1+\varepsilon)$\=/bi-Lipschitz equivalent to $B_{T_{x_0} M}(0,r)$ as soon as $\frac{r}{t}<\eta(\varepsilon)$. To conclude the proof it is hence enough to set $t_0\coloneqq r/\eta(L-1)$.
\end{proof}

Recall that an action on a measure space is \emph{essentially free} if the set of points with nontrivial stabilizer has measure zero. 

\begin{lem}\label{lem:small.singular.sets}
Let $\Gamma\curvearrowright M$ and $\Lambda\curvearrowright N$ be essentially free actions by isometries on compact Riemannian manifolds and let $L,A,r>0$ be fixed. If there exist increasing unbounded sequences $(t_k)$ and $(s_k)$, and  $(L,A)$\=/quasi\=/isometries $f_k\colon(M,\wdist{t_k})\to (N,\wLdist{s_k})$, then for every $k$ large enough, there exists a point $x_k\in M \setminus \chi_\Gamma^{t_k}(r)$ whose image $f_k(x_k)$ is \emph{not} in $\chi_\Lambda^{s_k}(Lr+A)$.
\end{lem}
\begin{proof}
Without loss of generality, we renormalize the Riemannian metrics so that $M$ and $N$ have volume $1$. Let $Y_k\subset (M,\wdist{t_k})$ be an $L(A+1)$\=/separated $L(A+1)$\=/dense subset. Note that the balls $B_{(M,\wdist{t_k})}\Big(y,L\frac{A+1}{2}\Big)$ with $y\in Y_k$ are disjoint, and (in the notation of Section \ref{ssec:small.balls.in.mflds}) they have volume bounded between $v_M\Big(L\frac{A+1}{2t_k}\Big)$ and $V_M\Big(L\frac{A+1}{2t_k}\Big)\abs*{B_\Gamma(L\frac{A+1}{2})}$.

Let $Z_k\subseteq Y_k$ be the subset of those points which are close to $\chi_\Gamma^{t_k}(r)$:
\[
Z_k\coloneqq \Big\{y\in Y_k\Bigmid \wdist{t_k}\big(y,\chi_\Gamma^{t_k}(r)\big)< L(A+1)\Big\}
\]
and let $\Omega_k\coloneqq N_{(M,\wdist{t_k})}\big(Z_k,L(A+1)\big)$. Then $\chi_\Gamma^{t_k}(r)$ is contained in $\Omega_k$ and $\Omega_k$ is contained in a ``small'' neighbourhood of $\chi_\Gamma^{t_k}(r)$:
\[
\Omega_k\subseteq N_{(M,\wdist{t_k})}\big(\chi_\Gamma^{t_k}(r),2L(A+1)  \big)
\subseteq B_\Gamma(2L(A+1))\cdot N_{(M,d)}\Big(\chi_\Gamma^{t_k}(r),2L\frac{A+1}{t_k}  \Big).
\]
Note that the measure of the right-hand side tends to $0$ as $k$ tends to infinity, because the sets $\chi_\Gamma^{t_k}(r)$ form a sequence of closed nested subsets that converge (in measure) to the union of the sets of fixed points of finitely many elements of $\Gamma$. 

Combining the two inequalities
\begin{align*}
&\abs{Z_k}v_M\big(L\frac{A+1}{2t_k}\big) \leq{\rm Vol}(\Omega_k)\to 0,\\
&\abs{Y_k}\abs*{B_\Gamma(L(A+1))}V_M\big(L\frac{A+1}{t_k}\big) \geq {\rm Vol}(M)=1
\end{align*}
with Lemma \ref{lem:volume.comparison.balls}, we obtain that the ratios $\abs{Z_k}/\abs{Y_k}$ tend to $0$ as $k$ goes to infinity.

Now, since $f_k$ is an $(L,A)$\=/quasi\=/isometry, the image $f_k(Y_k)$ is a $1$\=/separated $(L^2(A+1)+2A)$\=/dense subset of $(N,\wLdist{s_k})$ and we also have
\[
f_k\big(\chi_\Gamma^{t_k}(r)\big)\subseteq f_k(\Omega_k)\subseteq N_{(N,\wLdist{s_k})}\big(f_k(Z_k), L^2(A+1)+A \big).
\]
We deduce that the volume of (a neighborhood of) $f_k(\Omega_k)$ is bounded above by 
\begin{equation}\label{eqn:Z_k.lower.bound}
V_N\Big(\frac{L^2(A+1)+A}{s_k}\Big)\abs*{B_\Lambda(L^2(A+1)+A)}\abs{Z_k}.
\end{equation}

On the other hand, $1$\=/separatedness gives us an upper bound on $\abs{Y_k}$ in terms of $v_N$:
\begin{equation}\label{eqn:Y_k.upper.bound}
\abs{Y_k} v_N\Big(\frac{1}{2s_k}\Big)\leq {\rm Vol(N)}=1.
\end{equation}
Since $\abs{Z_k}/\abs{Y_k}$ tends to $0$ as $k$ tends to infinity, combining the estimates \eqref{eqn:Z_k.lower.bound} and \eqref{eqn:Y_k.upper.bound} and applying Lemma \ref{lem:volume.comparison.balls} implies that also the measure of (a neighborhood of) $f_k(\Omega_k)$ tends to $0$. As we also have that the volume of $\chi_\Lambda^{s_k}(Lr+A)\subseteq N$ tends to $0$, the statement of the lemma follows trivially.
\end{proof}

We can now prove the quasi\=/isometric rigidity result. The following theorem is inspired by \cite[Theorem 7]{MR3658731} and it directly implies Theorem \ref{thm:intro.qi.rigidity}.
\begin{thm} \label{thm:qi.rigidity.of.warped.cones}
Let $\Gamma\curvearrowright (M,d)$ be an essentially free action by isometries, $(t_k)$ an increasing diverging sequence, and $\Lambda$ a group generated by a finite set $S'$.
\begin{enumerate}[(i)] 
 \item If $\Lambda$ acts essentially freely by isometries on a compact Riemannian manifold $N$ and the sequence $\{(M,\wdist{t_k})\}_k$ is uniformly coarsely equivalent to $\{(N,\wLdist{s_k})\}$ for some increasing diverging sequence $(s_k)$, then $\Lambda\times\ZZ^{\dim(N)}$ is quasi\=/isometric to $\Gamma\times\ZZ^{\dim(M)}$.
 
 \item If there exists a sequence of finite index normal subgroups $\Lambda_k\vartriangleleft\Lambda$ with $\Lambda_{k+1}<\Lambda_k$ and $\bigcap_{k\in\NN}\Lambda_k=\{1\}$ and so that the sequence $\{{\rm Cay}(\Lambda/\Lambda_k,\pi_k(S'))\}_k$ of the Cayley graphs of the quotients is uniformly coarsely equivalent to $\{(M,\wdist{t_k})\}_k$, then $\Lambda$ is quasi\=/isometric to $\Gamma\times\ZZ^{\dim(M)}$.
\end{enumerate}
\end{thm}
\begin{proof}
  We first prove the first assertion. Suppose that the levels $(M,\wdist{t_k})$ and $(N,\wLdist{s_k})$ are uniformly coarsely equivalent. Then there exists a sequence of quasi\=/isometries $f_k\colon (M,\wdist{t_k})\to (N,\wLdist{s_k})$ that are all $(L,A)$\=/quasi\=/isometries for some fixed constants $L$ and $A$.
  
  Fix an integer radius $r\in\NN$. By Lemma \ref{lem:small.singular.sets},  for every $k$ large enough there exists a point $x_k\in M\setminus \chi_\Gamma^{t_k}(r)$ such that $f_k(x_k)$ is not in $\chi_\Lambda^{s_k}(Lr+A)$. Let $y_k\coloneqq f_k(x_k)$. Fix $\varepsilon>0$ small. By Lemma \ref{lem:balls.lipschitz.to.Euclidean}, we also have that there exists a $k=k(r)$ large enough so that  the balls $B_{(M,\wdist{t_k})}(x_{k(r)},r)$ and $B_{(N,\wLdist{s_k})}(y_{k(r)},Lr+A)$ are $(1+\varepsilon)$\=/bi-Lipschitz equivalent to $B_{\Gamma\times\RR^m}(r)$ and $B_{\Lambda\times\RR^n}(Lr+A)$ respectively, where $m=\dim(M)$ and $n=\dim(N)$.

Note that the inclusion $\ZZ^d\hookrightarrow \RR^d$ is a $(\sqrt d,\sqrt d)$\=/quasi\=/isometry and that the restriction of $f_k$ to $B_{(M,\wdist{t_k})}(x_k,r)$ is an $(L,A)$\=/quasi\=/isometric embedding into $B_{(N,\wLdist{s_k})}(y_k,Lr+A)$. We then have a concatenation of quasi\=/isometric embeddings as depicted in the following diagram:
  
  \
  
  \hspace{-4.5 ex}\begin{tikzpicture}
   \matrix(m)[matrix of math nodes, row sep=3em, column sep=3em, text height=1.5ex, text depth=0.25ex]
     {B_{\Gamma\times\ZZ^m}(r) 	& B_{\Gamma\times\RR^m}(r)            & B_{(M,\wdist{t_k})}(x_{k(r)},r)  \\
     \, B_{\Lambda\times\ZZ^n}\big(\sqrt{n}(Lr+A+1)\big)   & B_{\Lambda\times\RR^n}(Lr+A)        & B_{(N,\wLdist{s_k})}(y_{k(r)},Lr+A),  \\};
     
   \draw[right hook->] (m-1-1) -- (m-1-2) node[pos=0.5,above]{\scriptsize$(\sqrt{m},\sqrt{m})$};
   \draw[<->] (m-1-2) -- (m-1-3) node[pos=0.5,above]{\scriptsize$(1+\varepsilon,0)$};
   \draw[left hook->] (m-2-2) -- (m-2-1) node[pos=0.5,above]{\scriptsize$(\sqrt{n},\sqrt{n})$};
   \draw[<->] (m-2-2) -- (m-2-3) node[pos=0.5,above]{\scriptsize$(1+\varepsilon,0)$};
   \draw[right hook->] (m-1-3) -- (m-2-3) node[pos=0.5,right]{$f_{k(r)}$};
   \draw[dashed,right hook->] (m-1-1) -- (m-2-1) node[pos=0.5,right]{$\hat f_r$};
  \end{tikzpicture}

  \
  
  \noindent  where $\hat f_r$ is defined as the composition and the labels represent the quasi\=/isometry constants. Then $\hat f_r$ is a $(L',A')$\=/quasi\=/isometric embedding where $L'=\sqrt{nm}L$ and $A'=\sqrt{n}(\sqrt{m}L+A+1)$ (if the $\varepsilon$ coming from the bi-Lipschitz map is small enough, we can ignore it altogether because the distances in $\Gamma\times\ZZ^m$ and $\Lambda\times\ZZ^n$ take integer values).
  
  We thus obtained a sequence of uniform quasi\=/isometric embeddings $\hat{f}_r$. Note that by construction, $\hat{f}_r$  sends the identity element of $\Gamma\times\ZZ^m$ to the identity element of $\Lambda\times\ZZ^n$. It follows that for every fixed vertex $v\in\Gamma$, the image $\hat{f}_r(v)$ can only take finitely many values in $\Lambda\times\ZZ^n$ and hence there exists a subsequence $\hat{f}_{r_l}$ such that $\hat{f}_{r_l}(v)$ is constant.
  
 Using a diagonal argument, we can pass to a subsequence $\hat f_{r_i}$ such that for every $i>j$ the restriction of $\hat f_{r_i}$ to the ball $B_{\Gamma\times\ZZ^m}(j)$ coincides with $\hat f_{r_j}$. It follows that setting $\hat{f}|_{B_{\Gamma\times\ZZ^m}(i)}\coloneqq \hat{f}_{r_i}$ gives a well defined $(L',A')$\=/quasi\=/isometric embedding $\hat f\colon \Gamma\times\ZZ^m\to\Lambda\times\ZZ^n$.
  
  It only remains to show that $\hat f$ is coarsely surjective. This is easily done, because if $g$ is any quasi\=/isometry and $R$ is any radius, then there exists an $R'\geq R$ large enough so that the image $g\big(B(x,R')\big)$ is coarsely dense in $B\big(g(x),R\big)$. As $\hat f_{r}$ is defined as a composition of (restrictions of) quasi\=/isometries, it follows that  for every $R>0$ the image of $\hat f_r$ is coarsely dense in $B_{\Lambda\times\ZZ^n}(R)$ for every $r$ large enough and therefore the same holds true for $\hat f$.
  
  The proof of (ii) follows the same lines. Indeed, since the subgroups $\Lambda_i$ are a nested sequence with trivial intersection, for every $r\in\NN$ there is a $k=k(r)$ large enough so that the ball of radius $Lr+A$ in ${\rm Cay}(\Lambda/\Lambda_k, \pi_k(S'))$ and the ball of radius $Lr+A$ in ${\rm Cay}(\Lambda,S')$ are isometric. One can hence fix some points $x_k\in M\setminus \chi_\Gamma^{t_k}(r)$ and consider the diagram 
  \begin{center}
  \begin{tikzpicture}
   \matrix(m)[matrix of math nodes, row sep=3em, column sep=3em, text height=1.5ex, text depth=0.25ex]
     {B_{\Gamma\times\ZZ^m}(r) 	& B_{\Gamma\times\RR^m}(r)            & B_{(M,\wdist{t_k})}(x_{k(r)},r)  \\
     \, B_{\Lambda}(Lr+A)   &        & B_{\Lambda/\Lambda_k}(Lr+A)  \\};
     
   \draw[right hook->] (m-1-1) -- (m-1-2) node[pos=0.5,above]{\scriptsize$(\sqrt{m},\sqrt{m})$};
   \draw[<->] (m-1-2) -- (m-1-3) node[pos=0.5,above]{\scriptsize$(1+\varepsilon,0)$};
   \draw[<->] (m-2-1) -- (m-2-3) node[pos=0.5,above]{\scriptsize$\cong$};
   \draw[right hook->] (m-1-3) -- (m-2-3) node[pos=0.5,right]{$f_{k(r)}$};
   \draw[dashed,right hook->] (m-1-1) -- (m-2-1) node[pos=0.5,right]{$\hat f_r$};
  \end{tikzpicture}
  \end{center}
  and argue as above.
\end{proof}
\begin{rem}
 One can modify the proofs of Lemma \ref{lem:balls.lipschitz.to.Euclidean} and Theorem \ref{thm:qi.rigidity.of.warped.cones} in order to prove that, under the same hypotheses, if the warped cones $\CO_\Gamma(M) $ and $\CO_\Lambda(N)$ as defined in \cite{MR2115671} are quasi\=/isometric, then $\Gamma\times\ZZ^{\dim(M)+1}$ is quasi\=/isometric to $\Lambda\times\ZZ^{\dim(N)+1}$.
\end{rem}
It follows from the Splitting Theorem in \cite{MR1608566} that if $\Gamma$ and $\Lambda$ are cocompact lattices in a semisimple algebraic group with no rank\=/one simple factors and $\Gamma\times \ZZ^m$ is quasi\=/isometric to $\Lambda\times\ZZ^n$, then $m=n$ and $\Gamma$ is quasi\=/isometric to $\Lambda$.
\begin{cor}\label{cor:non.qi.expanders.from.SO}
 Any two expanders $\CG_{t_n}(\Gamma_d\curvearrowright\mathrm{SO}(d,\RR))$ and $\CG_{s_n}(\Gamma_{d'}\curvearrowright\mathrm{SO}(d',\RR))$ as in Theorem \ref{thm:expanders.from.SO} cannot be coarsely equivalent if $d\neq d'$. Moreover, the expanders $\CG_{t_n}(\Gamma_d\curvearrowright\mathrm{SO}(d,\RR))$ cannot be coarsely equivalent to $\CG_{s_n}(\Gamma_{d}\curvearrowright \sphere^{d-1})$.
\end{cor}

The class of cocompact lattices in a semisimple algebraic group over a non-Archimedean local field, in which every simple factor is of higher rank, is quasi-isometrically rigid \cite{MR1608566} (see also \cite{MR1434399}). It then follows from Theorem \ref{thm:qi.rigidity.of.warped.cones} that if $\CG_{t_k}(\Gamma\curvearrowright M)$ is uniformly quasi\=/isometric to a sequence $\Lambda/\Lambda_k$, then $\Lambda$ cannot be a cocompact lattice in such an algebraic group. Indeed, otherwise we would have that $\Gamma\times\ZZ^d$ must be itself a lattice in such a group, but this is not the case as it does not have property (T). This implies the following theorem.

\begin{thm}\label{thm:warped.cones.not.qi.to.Lafforgue}
 If $\Gamma$ is a group with Lafforgue strong Banach property (T) with an ergodic essentially free action by isometries on a Riemannian manifold $M$, then the graphs $\CG_{t_n}(\Gamma \curvearrowright M)$ form a superexpander that is not coarsely equivalent to a Lafforgue expander.
\end{thm}

The above result directly implies Theorem \ref{thm:infinitelymanysuperexpanders}.

\begin{rem}\label{rem:local.qi.rigidity}
 Theorem \ref{thm:qi.rigidity.of.warped.cones} can be used as in Corollary \ref{cor:non.qi.expanders.from.SO} to construct expanders that do not have coarsely equivalent subsequences, but the proof of the theorem is somehow implying that the expanders thus obtained are not even ``locally'' coarsely equivalent. For instance, if we assume the actions to be free we can avoid using Lemma \ref{lem:small.singular.sets} and then the proof of Theorem \ref{thm:qi.rigidity.of.warped.cones}.(i) works verbatim with the following (weaker) assumptions: there exist sequences of points $(x_k)$ in $M$ and $(y_k)$ in $N$, increasing unbounded sequences $(t_k)$ and $(s_k)$ in $[1,\infty)$, and neighbourhoods $A_k\subset(M,\wdist{t_k})$ and $B_k\subset(N,\wLdist{s_k})$ of $x_k$ and $y_k$, respectively, such that $A_k$ and $B_k$ are uniformly quasi\=/isometric and for every $r>0$ there exists a $k$ large enough so that the balls $B_{(M,\wdist{t_k})}(x_k,r)$ and $B_{(N,\wLdist{s_k})}(y_k,r)$ are contained in $A_k$ and $B_k$, respectively.
\end{rem}

\end{document}